\documentclass[a4paper,12pt]{scrartcl}
\usepackage[utf8x]{inputenc}
\usepackage{amsmath,amssymb,amsthm}
\usepackage{color}
\usepackage{hyperref}

\title{Prime polynomials in short intervals and in arithmetic progressions}
\author{Efrat Bank
\thanks{School of Mathematical Sciences, Tel Aviv University, Ramat Aviv, Tel Aviv 69978, Israel, \href{mailto:efratban@post.tau.ac.il}{efratban@post.tau.ac.il}}
\and  Lior Bary-Soroker \thanks{School of Mathematical Sciences, Tel Aviv University, Ramat Aviv, Tel Aviv 69978, Israel, \href{mailto:barylior@post.tau.ac.il}{barylior@post.tau.ac.il}}
\and Lior Rosenzweig
\thanks{Department of Mathematics, KTH,
SE-10044, Stockholm, Sweden, \href{mailto:lior.rosenzweig@gmail.com}{lior.rosenzweig@gmail.com}}}

\newtheorem{theorem}{Theorem}[section]
\newtheorem{conjecture}[theorem]{Conjecture}
\newtheorem{lemma}[theorem]{Lemma}
\newtheorem{proposition}[theorem]{Proposition}
\newtheorem{corollary}[theorem]{Corollary}
\theoremstyle{definition}

\newcounter{case}
\renewcommand{\thecase}{\Roman{case}}
\newenvironment{case}[1]{\refstepcounter{case}\par\medskip\noindent \textsc{Case}~\thecase. \emph{#1}. \par}{\medskip}

\newcommand{\disc}{\mathop{\rm disc}}
\newcommand{\Spec}{\mathop{\rm Spec}}

\newcommand{\ZZ}{\mathbb{Z}}

\newcommand{\FF}{\mathbb{F}}

\newcommand{\Gal}{\mathop{\rm Gal}}
\begin{document}

\maketitle

\begin{abstract}
In this paper we establish function field versions of two classical conjectures on prime numbers. The first  says that the number of primes in intervals $(x,x+x^\epsilon]$ is about $x^\epsilon/\log x$. The second says that the number of primes $p<x$ in the arithmetic progression $p \equiv a \pmod d$, for $d<x^{1-\delta}$, is about $\frac{\pi(x)}{\phi(d)}$, where $\phi$ is the Euler totient function. 

More precisely, for short intervals we prove: 
Let $k$ be a fixed integer. Then 
\[
\pi_q(I(f,\epsilon)) \sim \frac{\#I(f,\epsilon)}{k}, \qquad q\to \infty
\] 
holds uniformly for all prime powers $q$, degree $k$ monic polynomials $f\in \FF_q[t]$ and $\epsilon_0(f,q) \leq \epsilon$, where $\epsilon_0$ is either $\frac{1}{k}$, or $\frac{2}{k}$ if  $p\mid k(k-1)$, or $\frac{3}{k}$ if further $p=2$ and $\deg f'\leq 1$. Here $I(f,\epsilon)= \{g\in \FF_q[t] \mid \deg(f-g) \leq \epsilon \deg f\}$, and $\pi_q(I(f,\epsilon))$ denotes the number of prime polynomials in $I(f,\epsilon)$. 
We show that this estimation fails in the neglected cases. 

For arithmetic progressions we prove: Let $k$ be a fixed integer. Then 
\[
\pi_q(k;D,f) \sim \frac{\pi_q(k)}{\phi(D)}, \qquad q\to \infty,
\]
holds uniformly for all relatively prime polynomials $D,f\in \FF_q[t]$ satisfying $\|D\| \leq q^{k(1-\delta_0)}$, where $\delta_0$ is either $\frac{3}{k}$ or $\frac{4}{k}$  if  $p=2$ and $(f/D)'$ is a constant. Here $\pi_q(k)$ is the number of degree $k$ prime polynomials and $\pi_q(k;D,f)$ is the number of such polynomials in the arithmetic progression $P\equiv f\pmod D$.

We also generalize these results to arbitrary factorization types.
\end{abstract}

\section{Introduction}
We study two function field analogues of two classical problems in number theory concerning the number of primes in short intervals and in arithmetic progressions. We first introduce the classical problems. In the next sections we formulate the corresponding function field conjectures and the resolution of them in the limit $q\to \infty$. 

\subsection{Primes in short intervals}
Let $\pi(x)=\#\{ 0<p\leq x\mid p\mbox{ is a prime}\}$ be the prime counting function. By the Prime Number Theorem  (PNT)
\[
\pi(x) \sim \frac{x}{\log x}, \quad x\to \infty.
\]
Therefore, one may expect that an interval $I=(x,x+\Phi(x)]$ of size $\Phi(x)$ starting at a large $x$ contains about $\Phi(x)/\log x$ primes, i.e.\
\begin{equation}\label{eq:primes_short}
\pi(I) := \pi(x+\Phi(x)) - \pi(x) \sim \frac{\Phi(x)}{\log x}.
\end{equation}
From PNT \eqref{eq:primes_short} holds for $\Phi(x) \sim c x$, for any fixed $0<c<1$. By Riemann Hypothesis \eqref{eq:primes_short} holds for $\Phi(x) \sim  \sqrt {x} \log x$ or even $\Phi(x) \sim \epsilon \sqrt {x \log x}$ assuming a strong form of Montgomery's pair correlation conjecture  \cite{HeathBrownGoldston1984}. Concerning smaller powers of $x$ Granville conjectures \cite[p.~7]{Granville1995} 
\begin{conjecture}\label{conj:si}
If $\Phi(x)>x^\epsilon$ then \eqref{eq:primes_short} holds. 
\end{conjecture}
But even for $\Phi(x)=\sqrt{x}$ Granville says  \cite[p.~73]{Granville2010}: 
\begin{quote}
\emph{we know of no approach to prove that there are primes in all intervals $[x, x + \sqrt x]$.} 
\end{quote}

Heath-Brown \cite{HeathBrown1988Crelle}, improving Huxley \cite{Huxley1972inv}, proves Conjecture~\ref{conj:si} unconditionally, for $x^{\frac{7}{12}- \epsilon(x)}\leq \Phi(x)  \leq \frac{x}{\log^4x}$, where $\epsilon (x) \to 0$. 

We note that for extremely short intervals (e.g., for $\Phi(x) = \log x \frac{\log\log x \log\log\log\log x}{\log\log\log x}$)  \eqref{eq:primes_short}  fails \cite{Rankin1938} uniformly, but may hold for almost all $x$, see \cite{Selberg1943} and the survey \cite[Section~4]{Granville2010}.

\subsection{Primes in arithmetic progressions}
Let $\pi(x;d,a)$ denote the number of primes $p\leq x$ such that $p\equiv a\pmod d$. The Prime Number Theorem for arithmetic progressions says that if $a$ and $d$ are relatively prime and fixed, then 
\begin{equation}\label{eq:primesinAP}
\pi(x;d,a) \sim \frac{\pi(x)}{\phi(d)}, \quad x\to \infty,
\end{equation}
where $\pi(x)$ is the prime counting function and $\phi(d)$ is the Euler totient function, giving the number of positive integers $i$ up to $d$ with $\gcd(i,d)=1$.  

In many applications it is crucial to allow the modulus $d$ to grow with $x$. The interesting range is $d<x$ since if $d\geq x$,  there can be at most one prime in the arithmetic progression $p\equiv i\pmod d$. A classical conjecture is the following (for a slightly different form see \cite[Conjecture~13.9]{MontgomeryVaughan}).

\begin{conjecture}\label{conj:AP}
For every $\delta  >0$, \eqref{eq:primesinAP} holds in the range $d^{1+\delta}<x$. 
\end{conjecture}

Concerning results on this conjecture Granville says  \cite[p.~69]{Granville2010}:
\begin{quote}
\it $\ldots$ the best proven results have $x$ bigger than the exponential of a power of $q$ \emph{(Granville's $q$ is our $d$)} far larger than what we expect. If we are prepared to assume the unproven Generalized Riemann Hypothesis we do much better, being able to prove that the primes up to $q^{2+\delta}$ are equally distributed amongst the arithmetic progressions $\mod q$, for $q$ sufficiently large, though notice that this is still somewhat larger than what we expect to be true.
\end{quote}

In this work we establish function field analogues of Conjectures~\ref{conj:si} and \ref{conj:AP} for certain intervals of parameters $\epsilon,\delta$ which may be arbitrary small, and in particular breaking the barriers $\epsilon = 1/2$ in the former and $\delta =1$ in the latter. This indicates that Conjectures~\ref{conj:si} and \ref{conj:AP} should hold. A crucial ingredient is a type of Hilbert's irreducibility theorem over finite fields \cite{BarySoroker2012}.

\section{Function fields}
Let $\mathcal{P}_{\leq k}$ be the space of polynomials of degree at most $k$ over $\FF_q$ and $\mathcal{M}(k,q)\subseteq \mathcal{P}_{\leq k}$ the subset of monic polynomials of degree $k$. If $\deg f=k$, we let $\|f\| = q^{k}$. 

\subsection{Short intervals}
Let $\pi_q(k) = \# \{g\in M(k,q) \mid g \mbox{ is a prime polynomial}\}$ be the prime polynomial counting function. The Prime Polynomial Theorem (PPT) asserts that 
\[
\pi_q(k) = \frac{q^k}{k} + O\bigg(\frac{q^{k/2}}{k}\bigg).
\]

We replace the interval $[x,x+x^\epsilon)$ around $x$ with the interval $I$ around  $f\in \mathcal{M}(k,q)$ given by 
\[
I=I(f,\epsilon) =\{g\in\FF_q[t]\mid \|f-g\| \leq \|f\|^\epsilon\} = f + \mathcal{P}_{\leq m},
\]
where $m=\lfloor \epsilon\deg(f)\rfloor$. From this it is clear that it suffices to consider only $\epsilon = \frac{m}{\deg f}$, where $m$ is a nonnegative integer. 
If $\epsilon\geq1$, then $I(f,\epsilon)=\mathcal{P}_{\leq m}$, and so the PPT gives the number of primes there. Therefore, the interesting range is $\epsilon <1$, in which case we say that $I$ is a \emph{short interval around $f$}. In particular, $\mathcal{M}(k,q)=I(t^k,\frac{k-1}{k})$ is a short interval.
We note that all the polynomials in a short interval around a monic polynomial are monic.

For a short interval $I$ let $\pi_q(I) = \# \{g\in I \mid g \mbox{ is a prime polynomial}\}$. 
The naive analogue of Conjecture \ref{conj:si} says that $\pi_q(I(f,\epsilon))\sim \#I(f,\epsilon)/\deg f$ when $q^{\deg(f)}$ is large. However some anomalies can occur when both $\epsilon$ and $\deg f$ are small. For example if $\epsilon < \frac{1}{\deg f}$ this naive approximation fails, see Section~\ref{sec:pmidk}. Thus in the function field conjecture we add the assumption that $\epsilon$ is not too small when $\deg f$ is small:

\begin{conjecture}\label{conj:ffsi}
There exists a function $\epsilon_0(f,q)>0$ defined on $f\in \FF_q[t]$ such that $\displaystyle\lim_{\deg f\to \infty} \epsilon_0(f,q)=0$ and such that for any fixed $\epsilon$ the asymptotic formula 
\[
\pi_q(I(f,\epsilon)) \sim \frac{\# I(f,\epsilon)}{\deg(f)}, \qquad q^{\deg(f)}\to \infty,
\]
holds uniformly for all $q$, $f\in \FF_q[t]$ monic with $\epsilon_0(f,q) \leq \epsilon<1$. 
\end{conjecture}

\subsection{Primes in arithmetic progressions}
For relatively prime $f,D\in \FF_q[t]$ let \
\[
\pi_q(k;D ,f)=\#\{h=f+Dg\in \mathcal{M}(k,q)\mid h\mbox{ is a prime polynomial}\}.
\]
The Prime Polynomial Theorem for arithmetic progressions says that 
\begin{equation}\label{eq:PPTAP}
\pi_q(k;D ,f) = \frac{\pi_q(k)}{\phi(D)} + O\bigg(\frac{q^{k/2}}{k}\deg D\bigg).
\end{equation}
Here $\phi(D)$ is the function field Euler totient function, giving the number of units in $\FF_q[t]/D\FF_q[t]$. 

As in the classical case, we want to allow $\deg D$ to grow with $k$. The interesting range of parameters is $\deg D<k$ because if $\deg D\geq k$,  there is at most one monic prime in the arithmetic progression $h\equiv f\mod D$ of degree $k$. As in the short interval case, we must restrict the range  $\delta$ when $k$ is small. 

\begin{conjecture}\label{conj:ffAP}
There exists a function $\delta_0(f,D,q,k)$ defined over relatively prime $f,D\in \FF_q[t]$ such that $\displaystyle\lim_{k\to \infty}\delta_0(f,D,q,k) =0$  and such that for any fixed $\delta$ the asymptotic formula 
\[
\pi_q(k;D,f)\sim \frac{\pi_q(k)}{\phi(D)}, \qquad q^{k} \to \infty,
\]
holds uniformly for all $q$ and relatively prime polynomials $f,D\in \FF_q[t]$ satisfying $\deg D \leq k(1-\delta_0(f,D,q,k))$.
\end{conjecture}

(We replaced the range $d^{1+\delta}<x$ as in Conjecture~\ref{conj:AP} with $d<x^{1-\delta}$, for technical reasons.)

We note that
\[
\phi(D) \sim q^{\deg D}, \quad q\to \infty.
\]
Therefore, if $\deg D < \frac{k}{2}$, then \eqref{eq:PPTAP} gives that 
\[
\pi_q(k;D,f) \sim \frac{\pi_q(k)}{\phi(D)}, \quad q\to \infty.
\]
This range corresponds to $\delta>\frac{1}{2}$ in Conjecture~\ref{conj:ffAP}.
On the other hand \eqref{eq:PPTAP} gives nothing when $\delta\leq \frac{1}{2}$.

Partial results towards Conjectures~\ref{conj:ffsi} and \ref{conj:ffAP} in the limit $q\to \infty$ can be deduced from work of Cohen  \cite{Cohen1972} when the  characteristic of $\FF_q$ is greater than $\deg F$ and from the work of Keating and Rudnick \cite{KeatingRudnick2012} in an almost everywhere sense. 

We prove these conjectures in the limit $q\to \infty$ in general. 

\subsection{Results}
We settle both Conjectures~\ref{conj:ffsi} and \ref{conj:ffAP} in the limit $q\to \infty$. In fact, our method allows us to count polynomials with any given factorization type. Let us start by setting up the notation.

The degrees of the primes in the factorization of a polynomial $f\in \FF_q[t]$ to a product of prime polynomials gives a partition of $\deg f$, denoted by $\lambda_f$. Similarly, the lengths of the cycles in the factorization of a permutation $\sigma\in S_k$ to a product of disjoint cycles gives a partition of $k$, denoted by $\lambda_\sigma$. For a partition $\lambda$ of $k$ we denote the probability for $\sigma\in S_k$ to have $\lambda_\sigma=\lambda$ by 
\begin{equation}\label{eq:cycle_type}
P(\lambda)=\frac{\#\{ \sigma\in S_k\mid \lambda_\sigma = \lambda\}}{k!}.
\end{equation}

Let $k$ be a positive integer and $\lambda$ a partition of $k$. For a short interval $I$ around $f\in \mathcal{M}(k,q)$ we define the counting function
\[
\pi_q(I;\lambda) = \# \{g\in I \mid  \lambda_g=\lambda \}.
\]

\begin{theorem}\label{thm:mainpart}
Let $k$ be a positive integer. Then there exists a constant $c(k)>0$ depending only on $k$ such that for any 
\begin{itemize}
\item 
partition $\lambda$ of $k$,
\item
prime power $q=p^\nu$,
\item
short interval $I=f+\mathcal{P}_{\leq m}$, where $f\in \mathcal{M}(k,q)$ and $3\leq m<k$
\end{itemize}
we have
\[
\left|\pi_q(I;\lambda) - P(\lambda)q^{m+1}\right| \leq c(k)q^{m+\frac{1}{2}}
.
\]
We may take $m=1$ if $p\nmid k(k-1)$ and $m=2$ if $p\neq 2$ or if $\deg f'>1$.
\end{theorem}

For $q=p^{\nu}$ and $f\in \mathcal{M}(k,q)$, set
\[
\epsilon_0 (f,q)= \begin{cases} 
\frac{3}{k},& \mbox{ if $p=2$ and $\deg f'\leq 1$,}\\ 
\frac{1}{k}, & \mbox{ if $p\nmid k(k-1)$,}\\
\frac{2}{k}, & \mbox{ $p\mid k(k-1)$ .}
\end{cases}
\]
%
Applying Theorem~\ref{thm:mainpart} with the partition $\lambda$ consisting of one part, gives Conjecture~\ref{conj:ffsi} in the limit $q\to \infty$.
\begin{corollary}\label{cor:ffsi}
Let $k>0$ be fixed. The asymptotic formula 
\[
\pi_q(I(f,\epsilon)) \sim \frac{\#I(f,\epsilon)}{k}, \qquad q\to \infty
\] 
holds uniformly for all $q$, all $f\in \mathcal{M}(k,q)$, and all $ \epsilon\in [\epsilon_0(f,q),1)$.
\end{corollary}\label{cor:main}
In Section~\ref{sec:ce} we discuss the cases which are not included in Corollary~\ref{cor:ffsi}. This is done by studying the intervals $I(t^k,\epsilon)$ and showing that the Corollary~\ref{cor:ffsi} fails for $\epsilon<\epsilon_0$ in the cases where $p\neq 2$ or $\deg f'>1$. We do not know whether the corollary holds true in the remaining case. 

Next we discuss polynomials with given factorization type in arithmetic progressions:
For relatively prime $f,D\in \FF_q[t]$ with $D$ monic we define the counting function
\[
\pi_q(k;D,f;\lambda) = \#\{g \equiv f\pmod D\mid \deg g=k \mbox{ and } \lambda_g=\lambda\}. 
\]

We prove the following theorem for polynomials in arithmetic progressions. 
\begin{theorem}\label{thm:main2part}
Let $k$ be a positive integer. Then there exists a constant $c(k)>0$ depending only on $k$ such that for any 
\begin{itemize}
\item partition $\lambda$ of $k$,
\item
prime power $q=p^\nu$,
\item 
$D\in \FF_q[t]$  monic, such that $\deg D\leq k-4$,
\item 
$f\in \FF_q[t]$ relatively prime to $D$,
\end{itemize}
we have 
\[
\left|\pi_q(k;D,f;\lambda) - \frac{\pi_q(k;\lambda)}{\phi(D)}\right| \leq \frac{c(k)}{q^{\frac{1}{2}}} \cdot \frac{\pi_q(k;\lambda)}{\phi(D)}.
\]                                                      
Except when $p=2$ and $(f/D)'$ is constant, we may take $\deg D\leq k-3$.                                                  
\end{theorem}


In particular, when we consider the special case of $\lambda$ being the partition into one part, we get Conjecture~\ref{conj:ffAP} in the limit $q\to \infty$:

\begin{corollary}\label{cor:ffAP}
Let $k$ be a fixed integer. Then 
\[
\pi_q(k;D,f) \sim \frac{\pi_q(k)}{\phi(D)}, \qquad q\to \infty,
\]
holds uniformly for all relatively prime $D,f\in \FF_q[t]$ satisfying $\|D\| \leq q^{k(1-\delta_0)}$, where $\delta_0 = \frac{4}{k}$ if $(f/D)'$ is constant and $p=2$ and $\delta_0=\frac{3}{k}$ otherwise.
\end{corollary}

\section{Auxiliary results}
\subsection{Specializations}
We briefly recall some definitions and basic facts on specializations, see \cite[Section~2.1]{BarySoroker2012} for more details and proofs. Let
\begin{itemize}\renewcommand{\labelitemi}{}\itemsep0pt\parskip2pt 
\item $K$ be a field with algebraic closure $\tilde{K}$,
\item $\Gal(K)={\rm Aut}(\tilde{K}/K)$ the absolute Galois group of $K$,
\item $W=\Spec S$ and $V=\Spec R$ absolutely irreducible smooth affine $K$-varieties,
\item $\rho\colon W\to V$ a finite separable morphism which is generically Galois,
\item $F/E$ the function field Galois extension that corresponds to $\rho$,
\item $K$-rational point $\mathfrak{p}\in V(K)$ that is \'etale in $W$, and 
\item $\mathfrak{P}\in \rho^{-1}(\mathfrak{p})$. 
\end{itemize}
Then $\mathfrak{p}$ induces a homomorphism $\phi_{\mathfrak{p}}\colon R\to K$ that extends to a homomorphism $\phi_{\mathfrak{P}}\colon S\to \tilde{K}$ (via the inclusion $R\to S$ induced by $\rho$). Since $\mathfrak{p}$ is \'etale in $W$, we have a homomorphism $\mathfrak{P}^*\colon \Gal(K)\to \Gal(F/E)$ such that 
\begin{equation}\label{eq:geosol}
\phi_{\mathfrak{P}}(\mathfrak{P}^*(\sigma) (x)) = \sigma(\phi_{\mathfrak{P}}(x)), \quad \forall x\in S,\ \forall \sigma\in \Gal(K).
\end{equation}
For every other $\mathfrak{Q}\in \rho^{-1}(\mathfrak{p})$ there is $\tau\in \Gal(F/E)$ such that $\phi_{\mathfrak{Q}} = \phi_{\mathfrak{P}}\circ \tau$. Thus, by \eqref{eq:geosol}, $\mathfrak{Q}^* = \tau^{-1}\mathfrak{P}^*\tau$ and vice-versa every $\tau^{-1}\mathfrak{P}^*\tau$ comes from a point $\mathfrak{Q}\in \rho^{-1}(\mathfrak{p})$ . Hence $\mathfrak{p}^* = \{\mathfrak{Q}^* \mid \mathfrak{Q}\in \rho^{-1}(\mathfrak{p})\}$ is the orbit of $\mathfrak{P}^*$ under the conjugation action of $\Gal(F/E)$.

The key ingredients in the proof of the following proposition are the Lang-Weil estimates \cite[Theorem~1]{LangWeil1954} and the field crossing argument (as utilized in \cite[Proposition~2.2]{BarySoroker2012}).

\begin{proposition}\label{prop:irrsub}
Let $k$, $m$, and $B$ be positive integers, let $\lambda$ be a partition of $k$, let $\FF$ be an algebraic closure of $\FF_q$, and let $\mathcal{F}\in \FF_q[A_0, \ldots, A_{m}, t]$ be a polynomial that is separable in $t$ with  $\deg \mathcal{F}\leq B$ and $\deg_t \mathcal{F}=k$. Assume that 
\[
\Gal(\mathcal{F}, \FF(A_0, \ldots, A_{m})) = S_k.
\]
Denote by $N = N(\mathcal{F},q)$ the number of $(a_0, \ldots, a_{m}) \in \mathbb{F}_q^{m+1}$ such that $f=\mathcal{F}(a_0, \ldots, a_{m},t)$ has factorization type $\lambda_f=\lambda$. 
Then there is a constant $c(m,B)$ that depends only on $m$ and $B$ such that  
\[
\left|N-P(\lambda)q^{m+1}\right| \leq c(m,B) q^{m+1/2},
\]
where $P(\lambda)$ is defined in \eqref{eq:cycle_type}.
\end{proposition}

\begin{proof}
Let $\mathbf{A}=(A_0,\ldots, A_m)$ and $F$ the splitting field of $\mathcal{F}$ over  $\FF_q(\mathbf{A})$.
Since 
\[
S_k=\Gal(\mathcal{F}, \FF(\mathbf{A}))=\Gal(F\cdot\FF/\FF(\mathbf{A})) \leq \Gal(F/ \FF_q(\mathbf{A}))\leq S_k,
\]
all inequalities are in fact equalities and $\FF_q=F\cap \FF$. In particular, $\alpha\colon \Gal(F/\FF_q(\mathbf{A})) \to \Gal(F\cap \FF/\FF_q)=1$, so 
\begin{equation}\label{eq:ker}
\ker \alpha = S_k.
\end{equation}

Since $\Gal(\FF_q) =\langle \varphi \rangle \cong \hat{\ZZ}$ (with $\varphi$ being the  Frobenius map $x\mapsto x^q$) the homomorphisms $\theta\colon \Gal(\FF_q)\to S_k$ can be parametrized by permutations $\sigma \in S_k$. Explicitly, each $\sigma\in S_k$ gives rise to $\theta_\sigma\colon \Gal(\FF_q)\to S_k$ defined by $\theta_\sigma(\varphi)=\sigma$.
Let $\mathcal{C}$ be the conjugacy class of all permutations $\sigma$ with $\lambda_\sigma=\lambda$ and let $\Theta = \{ \theta_\sigma\mid \sigma\in \mathcal{C}\}$. Fix $\theta\in \Theta$. Clearly $\#\Theta=\#\mathcal{C}$, so by \eqref{eq:ker} we have 
\begin{equation}\label{eq:P}
\frac{\#\ker\alpha}{\#\Theta } = \frac{\#S_k}{\#C}=\frac{1}{P(\lambda)}.
\end{equation}

Let $Z$ be the closed subset of $\mathbb{A}^{m+1}=\Spec \FF_q[\mathbf{A}]$ defined by $D=\disc_t(\mathcal{F})= 0$ and $V=\mathbb{A}^{m+1}\smallsetminus Z =\Spec\FF_q[\mathbf{A},D^{-1}]$. By assumption $\mathcal{F}$ is separable in $t$, so $D$ is a nonzero polynomial of degree depending only on $B$. By \cite[Lemma~1]{LangWeil1954}, there exists a constant $c_1=c_1(m,B)$ such that 
\begin{equation}\label{eq:bddddd}
\#Z(K)\leq c_1 q^m.
\end{equation}
Let $u_1, \ldots, u_k$ be the roots of $\mathcal{F}$ in some algebraic closure of $\FF(A_0, \ldots, A_{m})$ and let $W =  \Spec \FF_q[u_1, \ldots, u_k, D^{-1}] \subseteq \mathbb{A}^{k+1}$. Then $W$ is an absolutely irreducible smooth affine $\FF_q$-variety of degree bounded in terms of $B=\deg \mathcal{F}$. The embedding $\FF_q[\mathbf{A},D^{-1}]\to \FF_q[u_1, \ldots, u_k, D^{-1}]$ induces a finite separable \'etale morphism $\rho\colon W\to V$. 

We apply \cite[Proposition~2.2]{BarySoroker2012} to get an absolutely irreducible smooth $\FF_q$-variety $\widehat{W}$ together with a finite separable \'etale morphism $\pi\colon \widehat{W}\to V$ with the following properties:
\begin{enumerate}\renewcommand{\theenumi}{\roman{enumi}}
\item\label{cond:11}
Let $U\subseteq V(\FF_q)$ be the set of $\mathfrak{p}\in V(\FF_q)$ that are \'etale in $W$ and such that $\mathfrak{p}^*=\Theta$. Then $\pi(\widehat{W}(\FF_q)) = U$. 
\item \label{cond:22}
For every $\mathfrak{p}\in U$,
\[
\#(\pi^{-1}(\mathfrak{p})\cap \widehat{W}(\FF_q)) = \frac{\#\ker\alpha}{\#\Theta } =\frac{1}{P(\lambda)}.
\]
(See \eqref{eq:P} for the last equality.)
\end{enumerate}
By the construction of $\widehat{W}$ in \emph{loc.\ cit.} it holds that $\widehat{W}_L=W_L$, for some finite extension $L/\FF_q$  (where subscript $L$ indicates the extension of scalars to $L$). Hence $\widehat{W}$ and $W$ have the same degree, which is bounded in terms of  $B$. Thus, by \cite[Theorem~1]{LangWeil1954}, there is a constant $c_2=c_2(m,B)$ such that 
\begin{equation}\label{eq:LWbd}
|\#\widehat{W}(\FF_q)-q^{m+1}| \leq c_2 q^{m+1/2}.
\end{equation}
Applying \eqref{cond:22} gives $P(\lambda)\cdot \#\pi(\widehat{W}(\FF_q))=\#\widehat{W}(\FF_q)$. So multiplying \eqref{eq:LWbd} by $P(\lambda)$ implies
\begin{equation}\label{eq:bdd}
|\#\pi(\widehat{W}(\FF_q))-P(\lambda)q^{m+1}| \leq P(\lambda)c_2 q^{m+1/2}\leq c_2q^{m+1/2}. 
\end{equation}

Since for $\mathfrak{p}=(a_0,\ldots, a_m)\in V(\FF_q)\subseteq \FF_q^{m+1}$ we have $\mathfrak{p}^* = \Theta$ if and only if the orbit type of $\mathfrak{p}^*$ is $\lambda$ (in the sense of \cite[p.~859]{BarySoroker2012}). Thus  $\lambda_{\mathcal{F}(a_0,\ldots, a_m,t)}=\lambda$ if and only 
$\mathfrak{p}^* = \Theta$ (\cite[Lemma~2.1]{BarySoroker2012}). Let 
\[
X = \{\mathfrak{p}=(a_0,\ldots, a_m)\in \FF_q^{m+1} \mid  \lambda_{\mathcal{F}(a_0,\ldots, a_m,t)}=\lambda \mbox{ and } D(a_0,\ldots, a_m)\neq 0\}.
\]
Then $N=\#X$. Equation \eqref{cond:11} gives $X\cap V(\FF_q) = \pi(\widehat{W}(\FF_q))$. Since $V=\mathbb{A}^{m+1} \smallsetminus Z$, it follows from   \eqref{eq:bddddd} and \eqref{eq:bdd}  that
\begin{eqnarray*}
\left|N-P(\lambda)q^{m+1}\right| &=& \left|\#X-P(\lambda)q^{m+1}\right|\\
				&=&\left|\#(X\cap V(\FF_q)) +\#(X\cap Z(\FF_q)) -P(\lambda)q^{m+1}\right|\\
				& \leq& \left|\#(X\cap V(\FF_q))-P(\lambda)q^{m+1}\right| + \#(X\cap Z(\FF_q)) \\
				&\leq &\left|\pi(\widehat{W}(\FF_q))-P(\lambda)q^{m+1}\right| + \#Z(\FF_q) \\
				&\leq & c_2q^{m+1/2} + c_1 q^{m} \leq c(m,B) q^{m+1/2},
\end{eqnarray*}
where $c=c_1+c_2$.
\end{proof}

\subsection{Calculating a Galois Group}
\begin{lemma}\label{lem:separableirreducible}
Let $F$ be an algebraically closed field, $\mathbf{A}=(A_0,\ldots, A_m)$ an $m$-tuple of variables with $m\geq 1$,  and $f,g\in F[t]$ relatively prime polynomials. Then $\mathcal{F}(\mathbf{A},t)=f(t)+g(t) \cdot (\sum_{i=0}^m A_it^i)$ is separable in $t$ and irreducible in the ring $F(\mathbf{A})[t]$. 
\end{lemma}

\begin{proof}
Since $\mathcal{F}$ is linear in $A_0$ and since $f,g$ are relatively prime, it follows that $\mathcal{F}$ is irreducible in $F[\mathbf{A},t]$,  hence by Gauss' lemma  also in $F(A)[t]$. Take $\alpha\in F$ with $g(\alpha)\neq 0$. Then  
\[
\mathcal{F}'(\alpha)=f'(\alpha) +g'(\alpha)(\sum_{i=0}^m A_i \alpha^i) + g(\alpha)A_1 + (\sum_{i=2}^m i A_i  \alpha^{i-1}) \neq 0,
\]
hence $\mathcal{F}'\neq 0$, so $\mathcal{F}$ is separable.
\end{proof}

\begin{lemma}\label{lem:dt}
Let $F$ be an algebraically closed field, $\mathbf{A}=(A_0,\ldots, A_m)$ an $m$-tuple of variables with $m\geq 2$,  and $f,g\in F[t]$ relatively prime polynomials with $\deg f>\deg g$.  The Galois group $G$ of $\mathcal{F}(\mathbf{A},t)=f(t)+g(t)\cdot (\sum_{i=0}^m A_it^i)$ over $F(\mathbf{A})$ is doubly transitive (with respect to the action on the roots of $\mathcal{F}$).
\end{lemma}

\begin{proof}
By replacing $t$ by $t+\alpha$, where $\alpha\in F$ is a root of $f$, we may assume that $f(0)=0$. Hence $f_0(t)=f(t)/t$ is a polynomial. By Lemma~\ref{lem:separableirreducible} the group $G$ is transitive. 
The image of $\mathcal{F}$ under the substitution $A_0=0$ is
\[
\bar{\mathcal{F}}=f(t)+g(t)\cdot \big(\sum_{i=0}^m A_it^i\big)=t\big( f_0(t) +g (t) \cdot\big(\sum_{i=1}^{m-1} A_it^{i-1}\big)\big).
\]
Lemma~\ref{lem:separableirreducible} then gives that $ f_0(t) +g (t) \cdot \big(\sum_{i=1}^{m-1}A_it^{i-1}\big) $ is separable and irreducible. This means that the stabilizer of the root $t=0$ in the Galois group of $\bar{\mathcal{F}}$ acts transitively on the other roots. But since $\bar{\mathcal{F}}$ is separable, its Galois group embeds into $G$, so the stabilizer of a root of $\mathcal{F}$ in $G$ is transitive. Thus $G$ is doubly transitive.
\end{proof}

For a rational function $\psi(t)\in F(t)$ the first and second Hasse-Schmidt derivatives of $\psi$ are denoted by $\psi'$ and $\psi^{[2]}$, respectively, and defined by 
\[
\psi(t+u) \equiv\psi(t) +\psi'(t)u+\psi^{[2]}(t)u^2 \mod u^3.
\]
A trivial observation is that $\psi'$ is the usual derivative of $\psi$ and, if the characteristic of $F\neq 2$, then $\psi^{[2]}=\frac{1}{2}\psi''$.

\begin{lemma}\label{lem:morse}
Let $\psi(t)\in F(t)$ be a rational function with $\psi^{[2]}$ nonzero and $A_1$ a variable. Then $\psi'(t)+A_1$ and $\psi^{[2]}(t)$ have no common zeros. 
\end{lemma}

\begin{proof}
This is obvious since the roots of $\psi'+A_1$ are transcendental over $F$, while those of $\psi^{[2]}$ are algebraic.
\end{proof}

\begin{lemma}\label{lem:excellent}
Let $F$ be an algebraically closed field of characteristic $p\geq 0$, $m\geq 2$, $\mathbf{A}=(A_1, \ldots, A_m)$, $f,g\in F[t]$ relatively prime polynomials and put $\psi=f/g$ and $\Psi = \psi +\sum_{i=1}^m A_it^i$. Assume $\deg f>\deg g+m$. Further assume that $\psi^{'}$ is not a constant if $p=m=2$. Then the system of equations
\begin{equation}\label{eq:excellentmorse}
\begin{array}{rcl}
\Psi'(\rho_1)&=&0\\
\Psi'(\rho_2)&=&0\\
\Psi(\rho_1)&=&\Psi(\rho_2)
\end{array}
\end{equation}
has no solution with distinct $\rho_1, \rho_2$ in an algebraic closure $\Omega$ of $F(\mathbf{A})$.
\end{lemma}

\begin{proof}
For short we write $\rho=(\rho_1,\rho_2)$.
Let
\[
-\varphi(t)=\bigg(\psi+\sum_{i=3}^m A_it^i\bigg)' =\psi'+\sum_{i=3}^m iA_it^{i-1}=\frac{f' g-fg'}{g^2} + \sum_{i=3}^miA_it^{i-1}.
\]
Then $\Psi'(t)=2A_2t+A_1-\varphi(t)$. If $m=2$, then $\varphi=-\psi'$, the latter being nonconstant if also $p=2$, by assumption. 

Let
\begin{align*}
c(\rho) &= \psi(\rho_1)-\psi(\rho_2) + \sum_{i=3}^m(\rho_1^{i}-\rho_2^{i})A_i\\ 
&= \Psi(\rho_1)-\Psi(\rho_2) - ((\rho_1^2-\rho_2^2)A_2+(\rho_1-\rho_2)A_1).
\end{align*}

The system of equations \eqref{eq:excellentmorse} defines an algebraic set $T\subseteq \mathbb{A}^2\times \mathbb{A}^{m}$ in the variables $\rho_1, \rho_2, A_1, \ldots, A_m$.
Let $\alpha\colon T\to \mathbb{A}^2$ and $\beta\colon T\to \mathbb{A}^m$ the projection maps. The system of equations \eqref{eq:excellentmorse} takes the matrix form 
\begin{equation}\label{eq:excellent}
M(\rho) \cdot \big(\begin{smallmatrix} A_2\\A_1\end{smallmatrix}\big) = B(\rho)=\Big(\begin{smallmatrix} \varphi(\rho_1)\\ \varphi(\rho_2)\\c(\rho)\end{smallmatrix}\Big), 
\end{equation}
where $M(\rho) = \left(\begin{smallmatrix}
2 \rho_1&1\\
2\rho_2&1\\
\rho_2^2-\rho_1^2&\rho_2-\rho_1
\end{smallmatrix}\right)$.
For every $\rho\in U=\{\rho\mid \rho_1\neq \rho_2,\ \varphi(\rho_i)\neq \infty, i=1,2\}$, the rank of $M(\rho)$ is $2$. Thus the dimension of the fiber $\alpha^{-1}(\rho)$, for any $\rho\in U$, is at most $m-2$.
Moreover, for a given $\rho\in U$, \eqref{eq:excellent} is solvable if and only if $\mathop{\rm rank}(M|B)=2$ if and only if $d(\rho)=\det(M|B)=0$. Thus, the solution space (restricting to $\rho\in U$) lies in $d(\rho)=0$. 

It suffices to prove that $d(\rho)$ is a nonzero rational function in the variables $\rho=(\rho_1,\rho_2)$. Indeed, this implies that  $\dim (\alpha(T))\leq \dim \{d(\rho)=0\} = 1$, so $\dim T \leq 1+m-2<m$. Thus $\beta(T)$ does not contain the generic point of $\mathbb{A}^m$ which is $\mathbf{A}=(A_0,\ldots, A_m)$ and hence \eqref{eq:excellentmorse} has no solution with $\rho\in \Omega^2$. 

A straightforward calculation gives 
\[
d(\rho) = (\rho_1-\rho_2) (2c(\rho) +(\rho_1-\rho_2)(\varphi(\rho_1)+\varphi(\rho_2))).
\]
If $m\geq 3$, then the coefficient of $A_3$ in $2c(\rho) +(\rho_1-\rho_2)(\varphi(\rho_1)+\varphi(\rho_2))$ is
\[
2(\rho_1^3-\rho_2^3)+3(\rho_1^2-\rho_2^2),
\]
which is nonzero in any characteristic and we are done. 

To this end assume $m=2$. If $p=2$, then $2c(\rho)=0$. Since $\varphi$ is not constant in this case, we have $\varphi(\rho_1)+ \varphi(\rho_2)\neq 0$ and we are done.  

Finally assume $m=2$ and $p\neq 2$. Then $c(\rho)=\psi(\rho_1)-\psi(\rho_2)$ and $\varphi=-\psi'$.  
We may assume without loss of generality that $f(0)=0$ (and hence $\psi(0)=0$).  Since $f(t)/t+g(t)(A_2t + A_1)$ is separable (Lemma~\ref{lem:separableirreducible}), we can replace $A_1$ and $A_2$ by $A_1+\alpha_1$ and $A_2+\alpha_2$, respectively, and $f$ by $f(t)+g(t)(\alpha_2t^2+\alpha_1t)$, for suitably chosen $\alpha_1,\alpha_2\in F$, to assume that $f(t)/t$ is separable. Since $\deg f>\deg g  +m\geq 2$, this implies that $f(t)$ has at least one simple root, say $\alpha$. Then $\alpha$ is a simple root of $\psi=f/g$. So $\psi'(\alpha)\neq 0$. Let $\beta\neq \alpha$ be another root of $f$, hence of $\psi$. 

If $\psi'(\beta)=0$, then we have $c(\alpha,\beta)=\psi(\alpha)-\psi(\beta)=0$, so 
\[
d(\alpha,\beta) =-(\alpha-\beta)^2 \psi'(\alpha)\neq 0
\]
and we are done. If $\psi'(\beta)\neq0$, then $\beta$ is a simple root of $\psi$, hence of $f$. But $\deg f> 2$, so there must be another root $\gamma$ of $\psi$. If $d=0$, then we must have 
\begin{align*}
\frac{d(\alpha,\beta)}{-(\alpha-\beta)^2}=0&= \psi'(\alpha) + \psi'(\beta)\\
\frac{d(\alpha,\gamma)}{-(\alpha-\gamma)^2}=0&= \psi'(\alpha) + \psi'(\gamma)\\
\frac{d(\gamma,\beta)}{-(\gamma-\beta)^2}=0&= \psi'(\gamma) + \psi'(\beta).
\end{align*}
So $2\psi'(\alpha)=0$. This contradiction implies that $d\neq 0$, as needed.
\end{proof}

\begin{proposition}\label{prop:Galoisgroup}
Let $F$ be a field of characteristic $p\geq 0$, let $1\leq m<k$, let $\mathbf{A}=(A_0, \ldots, A_{m})$ an $(m+1)$-tuple of variables, and let $f,g\in F[t]$ be relatively prime polynomials with $\deg g+m<k=\deg f$. Assume 
\begin{enumerate}
\item $2\leq m$ if $\deg g>0$,
\item $2\leq m$ if $p\mid k(k-1)$, and 
\item \label{eq:nonconstant} $(f/g)'$ is not constant if $p=m=2$.
\end{enumerate} 
Then the Galois group of $\mathcal{F}(\mathbf{A},t) = f(t)+g(t)\cdot (\sum_{i=0}^m A_i t^i)$ over $F(\mathbf{A})$ is 
\[
\Gal\left(\mathcal{F}, F(\mathbf{A})\right)=S_k.
\]
\end{proposition}

\begin{proof}
Let $\tilde{F}$ be an algebraic closure of $F$. Since $ \Gal(\mathcal{F},\tilde{F}(\mathbf{A}))\leq \Gal(\mathcal{F},F(\mathbf{A}))\leq S_k$, we may replace, without loss of generality, $F$ by $\tilde{F}$ to assume that $F$ is algebraically closed. 

If $p\nmid k(k-1)$ and $\deg g=0$, the result follows from  \cite[Theorem~1]{Cohen1980} (note that $F(A_0, \ldots, A_{m}) = F(A_2, \ldots, A_{m-1})(A_0,A_1)$, hence the result for $m=1$ in \emph{loc.\ cit.} extends to $m> 1$).

Assume that $2\leq m$. Then $G = \Gal(\mathcal{F},F(\mathbf{A}))\leq S_k$ is doubly transitive by Lemma~\ref{lem:dt}.

Let $\Omega$ be an algebraic closure of $F(A_1, \ldots, A_m)$ and consider the map $\Psi\colon \mathbb{P}^1_\Omega\to \mathbb{P}^1_\Omega$ defined locally by $t\mapsto -A_0 := \frac{f(t)}{g(t)}+\sum_{i=1}^m A_it^i$. The numerator of $\Psi'=\frac{f'g-g'f}{g^2} + \sum_{i=1}^m iA_it^i$ is 
\[
f'g-g'f + g^2 \cdot (\cdots + 2A_2t + A_1).
\]
If $m\geq 3$ or if $p\neq 2$, this numerator has positive degree. If $p=m=2$, then this numerator is $f'g-g'f+g^2A_1$, so it is not constant by \eqref{eq:nonconstant}. In any case, the numerator of $\Psi'$, hence $\Psi'$, has a root, say $\alpha\in \Omega$. Then $\Psi$ is ramified at $t=\alpha$. Lemma~\ref{lem:morse} says that the orders of ramifications are $\leq 2$, so the equation $\Psi(t)=\Psi(\alpha)$ has at most double roots in $\Omega$. Lemma~\ref{lem:excellent} says that the critical values are distinct, so $\Psi(t)=\Psi(\alpha)$ has at least $k-1$ solutions. But since $\alpha$ is a ramification point, the fiber over $\Psi(\alpha)$ is with exactly one double points. Hence the inertia group over $\Psi(\alpha)$ permutes two roots of 
\[
\mathcal{F}(\mathbf{A},t) = g(t)(\Psi(t)+A_0),
\]
and fixes the other roots (cf.\ \cite[Proposition 2.6]{BarySoroker2009Dirichlet}). In other words $G$ contains a transposition. Therefore, $G=S_k$ \cite[Lemma~4.4.3]{Serre2007Topics}. 
\end{proof}

\section{Proof of Theorem~\ref{thm:mainpart}}
Let $k$ be a positive integer, $\lambda$ a partition of $k$, $q=p^\nu$ a prime power,  $f\in \mathcal{M}(k,q)$, $3\leq m<k$ (or $1\leq m<k$ if $p\nmid k(k-1)$ or $2\leq m<k$ if $p\neq 2$ or $\deg f'>1$), and $I=f+\mathcal{P}_{\leq m}$.

Let $\FF$ be an algebraic closure of $\FF_q$.

Let $\mathcal{F}(A_0,\cdots,A_m,t)=f(t)+\sum_{i=0}^{m} A_it^i$. Then $\mathcal{F}$ satisfies the assumptions of Proposition~\ref{prop:Galoisgroup}, so $\Gal(\mathcal{F}, \FF(A_0, \ldots, A_m))=S_k$. 

Since $\deg \mathcal{F} =\deg_t \mathcal{F} =\deg f = k$ and $m<k$, by Proposition~\ref{prop:irrsub}, the number $N$ of $(a_0, \ldots, a_{m})\in \FF_q^{m+1}$ such that $f(t)+ \sum_{i=0}^{m} a_it^i$ has factorization type $\lambda$ satisfies
\[
\left|N - P(\lambda)q^{m+1}\right| \leq c(k)q^{m+1/2},
\]
where $c(k)>0$ is a constant depending only on $k$ (and not on $f$, $q$).  
This finishes the proof since by definition $N=\pi_{q}(I;\lambda)$.
\qed

\section{Proof of Theorem~\ref{thm:main2part}}
Let $k$ be a positive integer, $\lambda$ a partition of $k$, $q=p^\nu$ a prime power, 
 $D\in \FF_q[t]$ monic of $\deg D$ with $\deg D\leq k-3$ (or $\deg D\leq k-4$ if $p=2$ and $(f/D)'$ is constant), and $f\in \FF_q[t]$. 
Since we are interested in the number of prime polynomials in the arithmetic progression $g\equiv f\mod D$, we may replace $f$ by $f-QD$, for some polynomial $Q$ to assume  that $\deg f<\deg D$.  Let $m=\deg D$ and $\mathbb{F}$ be an algebraic closure of $\FF_q$.

Let 
\[
\mathcal{F}(\mathbf{A},t) = f(t)+D(t) \cdot \bigg(t^{m+1} + \sum_{i=0}^m A_it^i\bigg) = \tilde{f}(t) + D(t) \cdot\bigg( \sum_{i=0}^m A_it^i\bigg), \tilde{f} = f+D\cdot t^{m+1},
\] 
where $\mathbf{A}=(A_0,\ldots, A_m)$ is an $(m+1)$-tuple of variables. 
Since $\deg \tilde{f}=m+1+\deg D=k>\deg D+m$, Proposition~\ref{prop:Galoisgroup} gives that 
\[
\Gal(\mathcal{F}, \FF(\mathbf{A}))=S_k,
\]
 
Since $\deg \mathcal{F}=\deg_t\mathcal{F}=k$,
Proposition~\ref{prop:irrsub} implies that the number $N$ of $(a_0, \ldots, a_{m})\in \FF_q^{m+1}$ such that $f(t)+ D(t)\cdot (t^{m+1}+\sum_{i=0}^{m} a_it^i)$ has factorization type $\lambda$ satisfies
\[
\left|N - P(\lambda)q^{m+1}\right| \leq  c_1(k)q^{m+1/2},
\]
where $c(k)>0$ is a constant depending only on $k$ (and not on $f$, $q$).  

Finally, $\phi(D) = \|D\|\prod_{P\mid f}(1-1/\|P\|)$, where the products runs over the distinct prime polynomials $P$ dividing $D$. Since $\|P\|\geq q$ we have 
\[
\phi(D) = q^{\deg D}\left(1+O\left(\frac{1}{q}\right)\right)=q^{k-m-1} + O_k(q^{k-m-2}).
\]
By applying Theorem~\ref{thm:mainpart} to the interval $I(t^k,k-1)$, it follows that
\[
\pi_q(k;\lambda) = P(\lambda)q^{k} +O_k(q^{m+1/2}).
\]
Thus 
\[
\left|\frac{\pi_q(k;\lambda)}{\phi(D)}  -P(\lambda)q^{m+1}\right| \leq c_2(k) q^{m+1/2}
\]
and
\[
\bigg|N-\frac{\pi_{q}(k;\lambda)}{\phi(D)}\bigg| \leq  
\bigg| N - P(\lambda)q^{m+1}\bigg| + \bigg|\frac{\pi_{q}(k;\lambda)}{\phi(D)}-P(\lambda)q^{m+1} \bigg|\leq 
c(k)q^{m+1/2},
\]
where $c=c_1+c_2$. 
This finishes the proof since by definition $N=\pi_{q}(k;D,f;\lambda)$. 
\qed

\section{Small $\epsilon$}\label{sec:ce}
In this section we study the cases $\epsilon<\epsilon_0$ in Corollary~\ref{cor:main}, except for the case $p=m=2$, and $\deg f'\leq 1$ and show that the implication fails to hold in these cases. In the latter case we do not know whether the result holds or not. For the rest of the section let $m=\epsilon k$.

\subsection{$\epsilon<\frac{1}{k}$}

We denote Euler's totient function by $\phi(k) = |(\mathbb{Z}/k\mathbb{Z})^*|$. 

\begin{proposition}
For $k>1$ and $0<\epsilon<\frac{1}{k}$ we have 
\[
\pi_q(I(t^k,\epsilon)) = \pi_q(I(t^k,0))=
\begin{cases}
0,& q\not\equiv 1 \mod k\\
\frac{\phi(k)}{k}(q-1), &q\equiv 1\mod k.
\end{cases}
\]
In particular, if $k>2$, $|\pi_q(I(t^k,0))-q/k|\gg q$.
\end{proposition}

\begin{proof}
We separate the proof into cases.
\begin{case}{$\gcd(q,k)>1$}
In this case $t^k-a$ is inseparable for any $a\in \FF_p$. Since $\FF_q$ is perfect, this implies that $t^k-a$ is reducible. So $\pi_q(I(t^k,0))=0$. 
\end{case}

\begin{case}{$\gcd(q(q-1),k)=1$}\label{case:invertible}
In this case $k\neq 2$ and $1-q$ is invertible modulo $k$. 
Assume, by contradiction, that there exists $a\in \FF_q$ such that $f=t^k-a$ is irreducible in $\FF_p[t]$. Then the Frobenius map, $\varphi\colon x\mapsto x^q$, acts transitively on the roots of $f$. Thus $\alpha^q =  \zeta \alpha$, where $\zeta$ is a primitive $k$-th root of unity. We get that the orbit of $\alpha$ under $\varphi$ is
\[
\alpha \mapsto \alpha^q = \zeta\alpha \mapsto (\zeta\alpha)^q =\zeta^{1+q} \alpha \mapsto \cdots \mapsto \zeta^{1+q+\cdots+q^{k-1}} \alpha = \alpha.
\]
On the other hand, this orbit equals to the set of roots of $f$ which is $\{\zeta^i \alpha \mid i=0,\ldots, k-1\}$. So for every $i \mod k$ there is a unique $1\leq r\leq k$ such that 
\[
i\equiv 1+ q + \cdots + q^{r-1} \equiv (1-q)^{-1} (1-q^r) \pmod k.
\]
This is a contradiction since there are at most $\phi(k)<k$ powers of $q$ mod $k$, hence $\#\{(1-q)^{-1}(1-q^{r})\mod k\}<k = \#\{i\mod k\}$.
\end{case}
 
 \begin{case}{$\gcd(q,k)=1$ and $q\not\equiv 1\mod k$}
Let $g=\gcd(q-1,k)$; then $l=k/g>1$ and $\gcd(q(q-1),l)=1$. Let $a\in \FF_q$, and let $\alpha$ be a root of $f=t^k-a$. Then the polynomial $f_1=t^{l}-\alpha^l \in \FF_q[\alpha^l][t]$ is reducible by Case~\ref{case:invertible}. Since $\alpha$ is a root of $g$ and since $\alpha^l$ is a root of $f_2=t^g-a$, we get that 
\[
[\FF_q[\alpha]:\FF_q] = [\FF_q[\alpha]:\FF_q[\alpha^l]]\cdot [\FF_q[\alpha^l]:\FF_q]< l \cdot g =k.
\]
In particular, $f$ is reducible.  
 \end{case}

\begin{case}{$q\equiv 1\mod k$}
In this case $\FF_q$ contains a  primitive $k$-th root of unity. By Kummer theory $t^k - a$ is irreducible in $\FF_q$ if and only if the order  of $a(\FF_q^*)^k$ in $C=\FF_q^*/(\FF_q^*)^k$ is $k$. Since $\FF_q^*$ is cyclic of order $q-1$, the subgroup $C$ is also cyclic of order $k$. Hence, there are exactly $\phi(k)$ cosets of order $k$ in $C$. Each coset contains $\frac{q-1}{k}$ elements. So there are exactly $\frac{\phi(k)}{k}(q-1)$ prime polynomials $t^k-a$.
\end{case}
\end{proof}

\subsection{$\frac{1}{k}\leq \epsilon < \frac{2}{k}$ and $p\mid k$}\label{sec:pmidk}
In this case we study the interval $I(t^{p^2},\epsilon)=I(t^{p^2}, \frac{1}{k})=\{t^{p^2}-at+b\mid a,b\in \FF_q\}$ for $q=p^{2n}$. 

\begin{proposition}
For $q=p^{2n}$, $k=p^2$,   and $\frac{1}{k}\leq \epsilon < \frac{2}{k}$ we have
\[
\pi_q(I(t^{p^2},\epsilon))=0.
\]
In particular, $|\pi_q(I(t^p,\epsilon))-q^2/p|\gg q$. 
\end{proposition}

\begin{proof}
Let $F=\FF_{p^2}$, let $E$ be the splitting field of $\mathcal{F} = t^{p^2}-At+B$ over $K=\FF_q(A,B)$. Then, by \cite[Theorem 2]{Uchida1970}, 
\[
G=\Gal(\mathcal{F},F)\cong\Gal(E/F) \cong \Gal(E\cdot \FF, \FF(A,B)) \cong {\rm Aff}(F),
\]
as permutation groups.
Here $\FF$ is an algebraic closure of $\FF_q$ and ${\rm Aff}(F)$ is the group of transformation of the affine line $\mathbb{A}^1(F) = F$: 
\[
M_{c,d}\colon x\mapsto cx+d, \quad 0\neq c,d\in F.
\]
Since $|G|=p^2(p^2-1)$ and since the group of translation $T=\{x\mapsto x+d\} \cong\mathbb{F}_{p^2}$ is of order $p^2$, we get that $T$ is a $p$-sylow subgroup of $T$. But $T$ is of exponent $p$, hence there are no $p^2$-cycles in $G$.

For every $a,b\in \FF_q$, the Galois group $G_{a,b}$ of $f=t^{p^2}-at+b$ is a cyclic sub-quotient of $G$, hence of order $< p^2$. In particular $G_{a,b}$ acts intransitively on the roots of $f$, hence $f$ is reducible. 
\end{proof}

\subsection{$\frac{1}{k}\leq \epsilon < \frac{2}{k}$ and $p\mid k-1$} 
The details of this case are nearly identical to Section~\ref{sec:pmidk} with the distinction that the group ${\rm Aff}(F)$ is replaced by the group of transformations on the projective line, cf.\ \cite[Theorem 2]{Uchida1970}. Hence we state the result but omit the details. 

\begin{proposition}
For $q=p^{2n}$, $f=t^{p^2+1}$, $k=p^2+1$, and $\frac{1}{k}\leq \epsilon <\frac{2}{k}$  we have 
\[
\pi_q(I(t^{p^2+1},\epsilon)=0.
\]
\end{proposition}

\subsection*{Acknowledgments}
We thank Zeev Rudnick for helpful remarks on earlier drafts of this paper and for the suggestions to consider arithmetic progressions and different factorization types. We thank the referees for their many helpful comments. 

The first two authors were supported by a Grant from the GIF, the German-Israeli Foundation for Scientific Research and Development. The last author was supported by the G\"oran Gustafsson Foundation (KVA).

\bibliographystyle{plain}

\end{document}